\documentclass{article}
\usepackage{amsfonts}
\usepackage{makeidx}
\usepackage{amssymb}
\usepackage{amsmath}
\usepackage{amstext}
\usepackage{graphicx}
\usepackage{latexsym}
\usepackage{amsthm}

\newtheorem{theorem}{Theorem}

\newtheorem{corollary}[theorem]{Corollary}

\newtheorem{definition}[theorem]{Definition}
\newtheorem{example}[theorem]{Example}

\newtheorem{lemma}[theorem]{Lemma}

\newtheorem{proposition}[theorem]{Proposition}
\newtheorem{remark}[theorem]{Remark}

\input{tcilatex}
\begin{document}

\title{A topological approach to non-Archimedean mathematics}
\author{Vieri Benci\thanks{
Dipartimento di Matematica, Universit\`{a} degli Studi di Pisa, Via F.
Buonarroti 1/c, 56127 Pisa, ITALY; e-mail: \texttt{benci@dma.unipi.it }and
Centro Linceo Interdisciplinare "Beniamino Segre".} \and Lorenzo Luperi
Baglini\thanks{%
University of Vienna, Faculty of Mathematics, Oskar-Morgenstern-Platz 1,
1090 Vienna, AUSTRIA, e-mail: \texttt{lorenzo.luperi.baglini@univie.ac.at},
supported by grants P25311-N25 and M1876-N35 of the Austrian Science Fund
FWF.}}
\maketitle

\begin{abstract}
Non-Archimedean mathematics (in particular, nonstandard analysis) allows to
construct some useful models to study certain phenomena arising in PDE's;
for example, it allows to construct generalized solutions of differential
equations and variational problems that have no classical solution. In this
paper we introduce certain notions of non-Archimedean mathematics (in
particular, of nonstandard analysis) by means of an elementary topological
approach; in particular, we construct non-Archimedean extensions of the
reals as appropriate topological completions of $\mathbb{R}$. Our approach
is based on the notion of $\Lambda $-limit for real functions, and it is
called $\Lambda $-theory. It can be seen as a topological generalization of
the $\alpha $-theory presented in \cite{BDN2003}, and as an alternative
topological presentation of the ultrapower construction of nonstandard
extensions (in the sense of \cite{keisler}). To motivate the use of $\Lambda 
$-theory for applications we show how to use it to solve a minimization
problem of calculus of variations (that does not have classical solutions)
by means of a particular family of generalized functions, called
ultrafunctions.

\textbf{Keywords}: non-Archimedean mathematics, Nonstandard Analysis, limits
of functions, generalized functions, ultrafunctions.
\end{abstract}

\tableofcontents

\section{Introduction}

In a previous series of papers (\cite{ultra}, \cite{belu2012}, \cite%
{belu2013}, \cite{milano}, \cite{algebra}, \cite{beyond}, \cite{gauss}) we
have introduced and studied a new family of generalized functions called
ultrafunctions and its applications to certain problems in mathematical
analysis, including some applications to PDE's in \cite{gauss}. The
development of a rigorous study of (a large class of) PDE's in ultrafunction
theory is the object of \cite{Burgers}, where we exemplify our approach by
studying in detail Burgers' equation. Henceforth, it is our feeling that
many problems in PDE's theory could be fruitfully studied by means of the
theory of ultrafunctions.

However, one might have the impression that a drawback of our approach is
the use of the machinery of NSA, which is not a "common working tool" for
most analysts. Even if NSA has already been applied to many different fields
of mathematics (such as functional analysis, probability theory,
combinatorial number theory, mathematical physics and so on) to obtain
important results, the original formalism of Robinson, based on model theory
(see e.g. \cite{rob}), appears too technical to many researchers, and not
directly usable by most mathematicians. Since Robinson's work first
appeared, a simpler semantic approach (due to Robinson himself and Elias
Zakon) has been developed using the purely set-theoretic notion of
superstructure (see \cite{RZ}); we recall also the pioneering work by W.A.J.
Luxemburg (see \cite{Lux}), where a direct use of ultrapowers was made (see 
\cite{BDN2003}, \cite{BDNF06} for a complete presentation of alternative
simplified approaches to NSA). However, many researcher working in NSA have
the feeling that also these technical notions are not needed in order to
carry out calculus with actual infinitesimals, as well as to carry out
several applications of NSA. As a consequence, there have been many attempts
to simplify and popularize NSA by means of simplified presentations. We
recall here in particular the approaches of Henson \cite{Henson}, Keisler 
\cite{keisler} and Nelson \cite{nelson}; other attempts have been made by
Benci, Di Nasso and Forti with algebraic (see \cite{benci95}, \cite{benci99}%
, \cite{BDN2005}, \cite{F}) and topological approaches (see \cite{BDNF06}, 
\cite{DNF}). We also suggest \cite{loeb} where NSA is introduced in a
simplified way suitable for many applications. In our previous papers, we
tried to address the same issue by means of $\Lambda$-limits (see e.g. \cite%
{milano} for an axiomatic presentation of this approach to NSA). The basic
idea of $\Lambda$-limits is to present nonstandard objects as limits of
standard ones. However, in our previous works the word "limits" was not
intended in a topological sense: the "limits" where defined axiomatically
and no explicit topology was involved in the constructions.

The main aim of this paper is to show that, actually, $\Lambda$-limits can
be precisely characterized as topological limits. This approach will be
called $\Lambda$-theory; it allows to construct a topological approach to
NSA (related to but different from the approach of Benci, Di Nasso and Forti
in \cite{BDNF06}, \cite{DNF}) that, in our opinion, is well-suited for
researchers that are not experts in NSA and are interested to use certain
non-Archimedean arguments to study problems in analysis. In fact, it is our
feeling that presenting nonstandard constructions and results by means of a
topological approach might help such researchers to use them. For example,
we construct extensions of the reals (in the sense of NSA) as appropriate
topological completions of $\mathbb{R}$.

$\Lambda $-theory can be seen as a topological generalization of the $\alpha 
$-theory presented in \cite{BDN2003}. The idea behind our approach is to
embed $\mathbb{R}$ in particular Hausdorff topological spaces in which it is
possible to formalize the intuitive idea of hyperreals as topological limits
(in a sense that we will precise in Section \ref{OL}) of real functions.
From this point of view, our construction of the hyperreals starting from $%
\mathbb{R}$ shares some features with the construction of $\mathbb{R}$ as
the Cauchy completion of $\mathbb{Q}$. We also extend our construction to
define a topology on the superstructure $V(\mathbb{R})$ on $\mathbb{R}$,
that we use to define $\Lambda $-limits of bounded functions defined on $V(%
\mathbb{R}).$ Our construction is substantially equivalent to the
ultrapowers approach, and we will prove in Section \ref{comparison} that
within $\Lambda -$theory it is possible to construct a nonstandard universe
in the sense of \cite{keisler}. To motivate our feeling that $\Lambda$%
-theory can be fruitfully applied to study certain problems in Analysis, in
Section \ref{GS} we apply $\Lambda $-theory to solve a minimization problem
of calculus of variations that does not have classical solutions.

We want to remark that readers expert in NSA will easily recognize that $%
\Lambda $-theory is essentially equivalent to the ultrapowers construction
(we prove this fact in Section \ref{comparison}). Anyhow, in this paper, we
do not assume the knowledge of NSA by the reader.

\section{$\Lambda $-theory\label{lt}}

\subsection{The $\Lambda $-limit\label{OL}}

The only technical notion that we need to develop our approach to
non-Archimedean mathematics is that of ultrafilter:

\begin{definition}
Let $X$ be a set. An ultrafilter $\mathcal{U}$ on $X$ is a family of subsets
of $X$ that has the following properties:

\begin{enumerate}
\item $X\in \mathcal{U},$ $\emptyset \notin \mathcal{U};$

\item for every $A,B\subseteq X$ if $A\in \mathcal{U}$ and $A\subseteq B$
then $B\in \mathcal{U};$

\item for every $A,B\in \mathcal{U}$, $A\cap B\in \mathcal{U};$

\item for every $A\subseteq X$ we have that $A\in \mathcal{U}$ or $A^{c}\in 
\mathcal{U}.$
\end{enumerate}
\end{definition}

An ultrafiler $\mathcal{U}$ on $X$ is principal if there exists an element $%
x\in X$ such that $\mathcal{U}=\{A\subseteq X\mid x\in A\}.$ An ultrafilter
is free if it is not principal. From now on we let $\mathfrak{L}$ be an
infinite set equipped with a free ultrafilter $\mathcal{U}$. Every set $Q\in 
\mathcal{U}$ will be called \textbf{qualified set}. We will say that a
property $P$ is \textbf{eventually} true for the function $\varphi (\lambda
) $ if it is true for every $\lambda $ in a qualified set, namely if there
exists $Q\in \mathcal{U}$ such that $P(\varphi (\lambda ))$ holds for every $%
\lambda \in Q.$ We let $\Lambda \notin \mathfrak{L}$ and we consider the
space $\mathfrak{L}\cup \left\{ \Lambda \right\} .$ We equip $\mathfrak{L}%
\cup \left\{ \Lambda \right\} $ with a topology in which the neighborhoods
of $\Lambda $ are of the form $\left\{ \Lambda \right\} \cup Q,\ Q\in 
\mathcal{U}$. In this sense, one can imagine $\Lambda $ as being a "point at
infinity" for $\mathfrak{L}$ (in this sense, it plays a similar role to that
of $\alpha $ in the Alpha-Theory, see \cite{BDN2003}). With respect to this
topology, the notion of limit of a function at $\Lambda $ is specified as
follows:

\begin{definition}
\label{filter}Let $(X,\tau )$ be a Hausdorff topological space, let $%
x_{0}\in X$ and let $\varphi :\mathfrak{L}\rightarrow X$ be a function. We
say that $x_{0}$ is the $\Lambda $-limit of the function $\varphi $, and we
write%
\begin{equation}
\lim_{\lambda \rightarrow \Lambda }\varphi (\lambda )=x_{0},  \label{blu}
\end{equation}%
if for every neighborhood $V$ of $x_{0}$ the function $\varphi $ is
eventually in $V,$ namely if there is a qualified set $Q$ such that $\varphi
(Q)\subset V.$
\end{definition}

\begin{remark}
We use the notation $\lim_{\lambda \rightarrow \Lambda }\varphi (\lambda )$
since, as we already noticed, one may think to $\Lambda \notin \mathfrak{L}$
as a "point at $\infty $" and to the sets in $\mathcal{U}$ as neighborhoods
of $\Lambda $; it is conceptually similar to the point $\infty $ when one
considers $\mathbb{R}\cup \left\{ +\infty \right\} $. We prefer to use the
symbol $\Lambda $ rather than $\infty $ since one may think of $\Lambda $ as
a function of $\mathcal{U}$, namely $\Lambda =\Lambda \left( \mathcal{U}%
\right) $. Thus the explicit mention of $\Lambda $ is a reminder that the $%
\Lambda $-limit depends on $\mathcal{U}.$
\end{remark}

\begin{remark}
Another way to look at the limit (\ref{blu}) is to consider the Stone-\v{C}%
ech compactification $\beta \mathfrak{L}$ of $\mathfrak{L}$ with the
relative topology and to think to $\Lambda \in \beta \mathfrak{L}$ as of a
nontrivial element of this compactification.
\end{remark}

Limits as given by equation (\ref{blu}) will be called $\Lambda -$limits,
and we will call $\Lambda $-theory the approach to non-Archimedean
mathematics based on the notion of $\Lambda $-limit.

Our main result is the following:

\begin{theorem}
\label{nuovo}\textit{There exists a Hausdorff topological space }$\left( 
\mathbb{R}_{\mathfrak{L}},\tau \right) $ \textit{such that }

\begin{enumerate}
\item \label{uno}$\mathbb{R}_{\mathfrak{L}}=cl_{\tau }\left( \mathfrak{L}%
\times \mathbb{R}\right) ;$

\item \label{quattro}$\mathbb{R\subseteq R}_{\mathfrak{L}}\ $and $\forall
c\in \mathbb{R}$ 
\begin{equation*}
\lim_{\lambda \rightarrow \Lambda }\left( \lambda ,c\right) =c;
\end{equation*}

\item \label{tre}\textit{for every function }$\varphi :\mathfrak{L}%
\rightarrow \mathbb{R}$, the limit 
\begin{equation*}
\lim_{\lambda \rightarrow \Lambda }\left( \lambda ,\varphi (\lambda )\right)
\end{equation*}%
exists in $\left( \mathbb{R}_{\mathfrak{L}},\tau \right) $;

\item \label{sei}two functions are eventually equal if and only if%
\begin{equation*}
\lim_{\lambda \rightarrow \Lambda }\left( \lambda ,\varphi (\lambda )\right)
=\lim_{\lambda \rightarrow \Lambda }\left( \lambda ,\psi (\lambda )\right) .
\end{equation*}
\end{enumerate}
\end{theorem}

\begin{proof} We set%
\begin{equation*}
I=\left\{ \varphi \in \mathfrak{F}\left( \mathfrak{L},\mathbb{R}\right) \ |\
\varphi (x)=0\ \text{in a qualified set}\right\} .
\end{equation*}%
It is not difficult to prove that $I$ is a maximal ideal in $\mathfrak{F}%
\left( \mathfrak{L},\mathbb{R}\right) ;$ then%
\begin{equation*}
\mathbb{K}:=\frac{\mathfrak{F}\left( \mathfrak{L},\mathbb{R}\right) }{I}
\end{equation*}%
is a field. In the following, we shall identify a real number $c\in \mathbb{R%
}$ with the equivalence class of the constant function $\varphi:\mathfrak{L}\rightarrow\mathbb{R}$ such that $\varphi(\lambda)=c$ for every $\lambda\in \mathfrak{L}$.\\
We set 
\begin{equation*}
\mathbb{R}_{\mathfrak{L}}=\left( \mathfrak{L}\times \mathbb{R}\right) \cup 
\mathbb{K}\text{.}
\end{equation*}%
We equip $\mathbb{R}_{\mathfrak{L}}$ with the following topology $\tau $. A
basis for $\tau $ is given by 
\begin{equation*}
b(\tau )=\left\{ N_{\varphi ,Q}\ |\ \varphi \in \mathfrak{F}\left( \mathfrak{%
L},\mathbb{R}\right) ,Q\in \mathcal{U}\right\} \cup \mathcal{P}(\mathfrak{L}%
\times \mathbb{R})
\end{equation*}

where%
\begin{equation*}
N_{\varphi ,Q}:=\left\{ \left( \lambda ,\varphi (\lambda )\right) \mid
\lambda \in Q\right\} \cup \left\{ \left[ \varphi \right] _{I}\right\}
\end{equation*}%
is a neighborhood of $\left[ \varphi \right] _{I}$ for every $Q\in \mathcal{U%
}.$

In order to show that $b(\tau )$ is a basis for a topology, we have to show
that 
\begin{equation*}
\forall A,B\in b(\tau )\ \forall x\in A\cap B \ \exists C\in b(\tau )\text{ such that } x\in C\subset
A\cap B.
\end{equation*}%
Let $A,B\in b(\tau ).$ Let $x\in A\cap B$. If $x\notin \mathbb{K}$ then we can just set $C=A\cap B\cap \mathfrak{L}\times\mathbb{R}$, as the topology is discrete on $\mathfrak{L}\times\mathbb{R}$. If $x\in\mathbb{K}$ then 
there exist $R,S\in \mathcal{U}$ such that $A=N_{\varphi ,R}$ and $B=N_{\psi
,S}$ with $\left[ \varphi \right] _{I}=\left[ \psi \right] _{I}=x .$ Hence
there exists $Q\in \mathcal{U}$ such that 
\begin{equation*}
\forall \lambda \in Q,\ \varphi (\lambda )=\psi (\lambda ).
\end{equation*}%
Thus if we set $C:=N_{\varphi ,R\cap S\cap Q}$ we have that $x\in C\subset A\cap
B.$

Let us show that $\tau $ is a Hausdorff topology. Clearly it is sufficient
to check it for points in $\mathbb{K}$, so let $\xi \neq \zeta \in \mathbb{K}
$. Since $\xi \neq \zeta $, there exists $\varphi ,\psi \in \mathfrak{F}%
\left( \mathfrak{L},\mathbb{R}\right), Q\in \mathcal{U}$ such that%
\begin{equation*}
\xi =\left[ \varphi \right] _{I},\ \zeta =\left[ \psi \right] _{I}\ \ \text{%
and}\ \ \forall \lambda \in Q,\ \varphi (\lambda )\neq \psi (\lambda ).
\end{equation*}%
Therefore%
\begin{equation*}
N_{\varphi ,Q}\cap N_{\psi ,Q}=\varnothing .
\end{equation*}

Let us observe that, by construction, for every function\textit{\ }$\varphi :%
\mathfrak{L}\rightarrow \mathbb{R}$ we have that%
\begin{equation}
\lim_{\lambda \rightarrow \Lambda }\left( \lambda ,\varphi (\lambda )\right)
=\left[ \varphi \right] _{I}.  \label{elga}
\end{equation}%
In fact, given a neighborhood $N_{\varphi ,Q}$ of $\left[ \varphi \right]
_{I},$ we have that $\{\varphi (\lambda )\mid \lambda \in Q\}\subseteq
N_{\varphi ,Q},$ so $\left[ \varphi \right] _{I}$ is a $\Lambda $-limit of
the function $(\lambda ,\varphi (\lambda )).$ Since the space is Hausdorff,
the limit is unique, so $\lim_{\lambda \rightarrow \Lambda }\left( \lambda
,\varphi (\lambda )\right) =\left[ \varphi \right] _{I}.$

Let us prove that $\left( \mathbb{R}_{\mathfrak{L}},\tau \right) $ has the
desired properties:

\begin{itemize}
\item property (\ref{uno}) follows directly by the definition of $\tau ;$

\item property (\ref{quattro}) follows by the identification of every real
number $c\in \mathbb{R}$ with the equivalence class of the constant function 
$\left[ c\right] _{I};$

\item properties (\ref{tre}) and (\ref{sei}) follow by equation (\ref{elga}).\qedhere

\end{itemize}
\end{proof}

\begin{remark}
In \cite{BDNF06} and \cite{DNF}, nonstandard extensions are constructed by
means of similar, but different, topological considerations based on the
choice of the ultrafilter $\mathcal{U}$. However the authors showed (see
Theorem 4.5 in \cite{DNF}) that such extensions are Hausdorff if and only if
the ultrafilter $\mathcal{U}$ is Hausdorff (see again \cite{DNF}, Section 4
and 6), and in \cite{BS} Bartoszynski and Shelah proved that it is
consistent with ZFC that there are no Hausdorff ultrafilters. By contrast,
in our topological approach the extensions are always constructed inside
Hausdorff topological spaces under the much milder request of $\mathcal{U}$
being fine. This is possible because we incorporate the set of indices $%
\mathfrak{L}$ in the space.
\end{remark}

Motivated by the philosophical similarity between the properties expressed
in Thm \ref{nuovo} and the construction of $\mathbb{R}$ as the Cauchy
completion of $\mathbb{Q}$, we introduce the following definition:

\begin{definition}
\textit{A Hausdorff topological space }$\left( \mathbb{R}_{\mathfrak{L}%
},\tau \right) $ that satisfies conditions (\ref{uno})-(\ref{sei}) of Thm %
\ref{nuovo} will be called a $(\mathfrak{L},\mathcal{U})$-\textbf{completion
of} $\mathbb{R}$.
\end{definition}

\subsection{The hyperreal field\label{HF}}

Let $\left( \mathbb{R}_{\mathfrak{L}},\tau \right) $ be a $(\mathfrak{L},%
\mathcal{U})$-completion of $\mathbb{R}$. Let us fix some notation: we will
denote by $\mathbb{K}$ the set%
\begin{equation*}
\mathbb{K}=\left\{ \lim_{\lambda \rightarrow \Lambda }\left( \lambda
,\varphi (\lambda )\right) \mid \varphi \in \mathfrak{F}\left( \mathfrak{L},%
\mathbb{R}\right) \right\} .
\end{equation*}%
The aim of this section is to study the basic properties of $\mathbb{K}$.

\begin{proposition}
\label{carattere}$\left( \mathfrak{L}\times \mathbb{R}\right) \cap \mathbb{K}%
=\emptyset$.
\end{proposition}

\begin{proof} Let us suppose by contrast that there exists $\varphi :\mathfrak{L\rightarrow }X$ such
that 
\begin{equation*}
\lim_{\lambda \rightarrow \Lambda }\left( \lambda ,\varphi (\lambda )\right)
=(\lambda _{0},r)\in \mathfrak{L}\times \mathbb{R}\text{.}
\end{equation*}%
Since $\left\{ \left( \lambda _{0},r\right) \right\} $ is open, by
definition there exists $Q\in \mathcal{U}$ such that $\forall \lambda \in
Q,~\left( \lambda ,\varphi (\lambda )\right) =(\lambda _{0},r).$ Therefore $%
Q=\{\lambda _{0}\}$, and this is absurd since $\mathcal{U}$ is free. \end{proof}

From condition (\ref{uno}) in Thm \ref{nuovo} we know that $(\mathfrak{L}%
\times \mathbb{R})\uplus\mathbb{K}\subseteq \mathbb{R}_{\mathfrak{L}}.$ In
general, this inclusion might be proper; henceforth we introduce the
following definition:

\begin{definition}
We say that $(\mathbb{R}_{\mathfrak{L}},\tau)$ is a minimal $(\mathfrak{L},%
\mathcal{U})$-completion of $\mathbb{R}$ if $\mathbb{R}_{\mathfrak{L}}=(%
\mathfrak{L}\times \mathbb{R})\uplus\mathbb{K}.$
\end{definition}

It is immediate to see that any $(\mathfrak{L},\mathcal{U})$-completion of $%
\mathbb{R}$ contains a minimal $(\mathfrak{L},\mathcal{U})$-completion of $%
\mathbb{R}$, and that any minimal $(\mathfrak{L},\mathcal{U})$-completion of 
$\mathbb{R}$ does not properly contain another minimal $(\mathfrak{L},%
\mathcal{U})$-completion of $\mathbb{R}$ (and this is what motivates the
choice of the name "minimal" for such extensions).

From now on we will be only interested in minimal $(\mathfrak{L},\mathcal{U}%
) $-completions. \FRAME{ftbpFOX}{4.7305in}{1.919in}{0pt}{\Qct{\textit{%
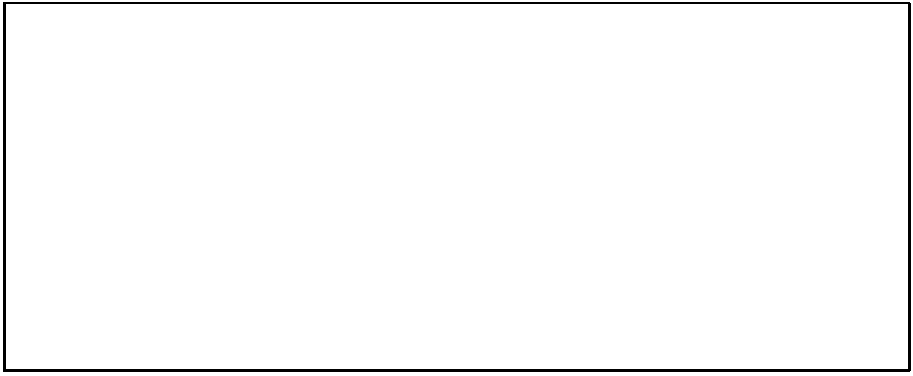
Representation of the sets} $\mathfrak{L}\times \mathbb{R}$, $\mathbb{R}_{%
\mathfrak{L}}=cl_{\protect\tau }\left( \mathfrak{L}\times \mathbb{R}\right) $
and $\mathbb{K}=\mathbb{R}^{\ast }$}}{}{immagine04.png}{\special{language
"Scientific Word";type "GRAPHIC";maintain-aspect-ratio TRUE;display
"USEDEF";valid_file "F";width 4.7305in;height 1.919in;depth
0pt;original-width 4.984in;original-height 2.0044in;cropleft "0";croptop
"1";cropright "1";cropbottom "0";filename 'Immagine04.png';file-properties
"XNPEU";}}By condition (\ref{quattro}) in the definition of $(\mathfrak{L},%
\mathcal{U})$-completions it follows that $\mathbb{R}\subseteq \mathbb{K}$.
Moreover we have the following result:

\begin{proposition}
For every finite subset $F\subseteq \mathbb{R}$, for every function $\varphi
:\mathfrak{L}\rightarrow F$ we have that%
\begin{equation*}
\lim_{\lambda \rightarrow \Lambda }\left( \lambda ,\varphi (\lambda )\right)
\in F.
\end{equation*}
\end{proposition}

\begin{proof} Let $F=\{x_{1},...,x_{n}\}.$ For every $i\leq n$ let 
\begin{equation*}
A_{i}=\{\lambda \in \mathfrak{L}\mid \varphi (\lambda )=x_{i}\}.
\end{equation*}%
Since $\mathcal{U}$ is an ultrafilter, there exists exactly one index $%
i_{0}\leq n$ such that $A_{i_{0}}\in \mathcal{U}$. Now let $\xi
=\lim_{\lambda \rightarrow \Lambda }\left( \lambda ,\varphi (\lambda
)\right) .$ Let us suppose that $\xi \neq x_{i_{0}}$. Let $O_{1},O_{2}$ be
disjoint open sets such that $\xi \in O_{1},x_{i_{0}}\in O_{2}$. Since $%
x_{i_{0}}$ is the limit of the constant function with value $x_{i_{0}}$,
there exists $B\in \mathcal{U}$ such that 
\begin{equation*}
\{(\lambda ,x_{i_{0}})\mid \lambda \in B\}\subseteq O_{2.}
\end{equation*}
Let $C\in \mathcal{U}$ be such that $\{(\lambda ,\varphi (\lambda ))\mid
\lambda \in C\}\subseteq O_{1}.$ Then by construction we have that 
\begin{equation*}
\forall \lambda \in A_{i_{0}}\cap B\cap C\text{ }(\lambda ,\varphi (\lambda
))=(\lambda ,x_{i_{0}})\in O_{1}\cap O_{2\text{,}}
\end{equation*}%
and this is a contradition since $O_{1}\cap O_{2}=\emptyset .$ Therefore $%
\lim_{\lambda \rightarrow \Lambda }\left( \lambda ,\varphi (\lambda )\right)
=x_{i_{0}}\in F$. \end{proof}

There is a natural way to define sums and products of elements of $\mathbb{K}
$:

\begin{definition}
We set 
\begin{eqnarray*}
\lim_{\lambda \rightarrow \Lambda }\left( \lambda ,\varphi (\lambda )\right)
+\lim_{\lambda \rightarrow \Lambda }\left( \lambda ,\psi (\lambda )\right)
&:&=\lim_{\lambda \rightarrow \Lambda }\left( \lambda ,\varphi (\lambda
)+\psi (\lambda )\right) ; \\
\lim_{\lambda \rightarrow \Lambda }\left( \lambda ,\varphi (\lambda )\right)
\cdot \lim_{\lambda \rightarrow \Lambda }\left( \lambda ,\psi (\lambda
)\right) &:&=\lim_{\lambda \rightarrow \Lambda }\left( \lambda ,\varphi
(\lambda )\cdot \psi (\lambda )\right) .
\end{eqnarray*}
\end{definition}

\begin{theorem}
$(\mathbb{K},+,\cdot ,0,1)$ is a field which contains $\mathbb{R}$.
\end{theorem}

\begin{proof} That $\mathbb{R\subseteq K}$ follows by condition (\ref%
{quattro}) of the definition of $(\mathfrak{L},\mathcal{U})$-completion. The
only non trivial property that we have to prove to show that $\mathbb{K}$ is
a field is the existence of a multiplicative inverse for every $x\neq 0$.
Let $x\in \mathbb{K}$, $x\neq 0.$ Since the topology is Hausdorff and $x\neq
0,$ there is a set $Q\in \mathcal{U}$ such that 
\begin{equation*}
\forall \lambda \in Q,\text{ }\varphi (\lambda )\neq 0.
\end{equation*}

Let $\phi :\mathfrak{L}\rightarrow \mathbb{R}$ be defined as follows: 
\begin{equation*}
\phi (\lambda )=\left\{ 
\begin{array}{cc}
1 & \text{if }\lambda \notin Q; \\ 
\frac{1}{\varphi (\lambda )} & \text{if }\lambda \in Q.%
\end{array}%
\right.
\end{equation*}%
Then $\varphi (\lambda )\cdot \phi (\lambda )=1$ for every $\lambda \in Q,$
thus $\lim_{\lambda \rightarrow \Lambda }\left( \lambda ,\varphi (\lambda
)\right) \cdot \lim_{\lambda \rightarrow \Lambda }\left( \lambda ,\phi
(\lambda )\right) =\lim_{\lambda \rightarrow \Lambda }\left( \lambda
,\varphi (\lambda )\cdot \phi (\lambda )\right) =1,$ namely%
\begin{equation*}
x^{-1}:=\lim_{\lambda \rightarrow \Lambda }\left( \lambda ,\phi (\lambda
)\right)
\end{equation*}%
is the inverse of $x.$ \end{proof}

The ordering of $\mathbb{R}$ can be extended to $\mathbb{K}$ by setting 
\begin{equation}
\lim_{\lambda \rightarrow \Lambda }\ \left( \lambda ,\varphi (\lambda
)\right) <\lim_{\lambda \rightarrow \Lambda }\ \left( \lambda ,\psi (\lambda
)\right) \Leftrightarrow \varphi (\lambda )<\psi (\lambda )\ \text{%
eventually,}  \label{pucci}
\end{equation}%
namely iff $\{(\lambda,\varphi(\lambda)-\psi(\lambda))\mid
\varphi(\lambda)-\psi(\lambda)\geq 0\}\cup [\varphi-\psi]$ is open (i.e. iff 
$\{\lambda \in \mathfrak{L\mid }\varphi (\lambda )<\psi (\lambda )\}$ is
qualified). This ordering is clearly an extension of the ordering relation
defined on $\mathbb{R}$ since, for every $x,y\in \mathbb{R}$, if $x\leq y$
and $\varphi _{x},\varphi _{y}:\mathfrak{L}\rightarrow \mathbb{R}$ are the
constant sequences with values resp. $x,y$ then 
\begin{equation*}
\{\lambda \in \mathfrak{L\mid }\varphi _{x}(\lambda )<\varphi _{y}(\lambda
)\}=\mathfrak{L},
\end{equation*}%
which is qualified.

\begin{remark}
Usually, the inclusion $\mathbb{R\subseteq K}$ is proper: e.g., let $%
\mathcal{U}$ be a countably incomplete ultrafilter\footnote{%
An ultrafilter $\mathcal{U}$ is countably incomplete if there exists a
family $\langle A_{n}\mid n\in \mathbb{N}\rangle $ of elements of $\mathcal{U%
}$ such that $\bigcap_{n\in \mathbb{N}}A_{n}=\emptyset .$}. Let $\langle
A_{n}\mid n\in \mathbb{N}\rangle $ be a family of elements of $\mathcal{U}$
such that $\bigcap_{n\in \mathbb{N}}A_{n}=\emptyset ,$ let $%
B_{n}=\bigcap_{i\leq n}A_{n}$ for all $n\in \mathbb{N}$ and let $\phi :%
\mathfrak{L}\rightarrow \mathbb{R}$ be defined as follows: for every $%
\lambda \in \mathfrak{L}$, 
\begin{equation*}
\phi (\lambda )=n\Leftrightarrow \lambda \in B_{n}\setminus B_{n+1.}
\end{equation*}%
Then $\lim_{\lambda \rightarrow \Lambda }\ \left( \lambda ,\phi (\lambda
)\right) \notin \mathbb{R}$: in fact, $\lim_{\lambda \rightarrow \Lambda }\
\left( \lambda ,\phi (\lambda )\right) >n$ for every $n\in \mathbb{N}$ (and
so, in particular, this limit is infinite). This holds since, for every $%
n\in \mathbb{N}$, by construction we have that%
\begin{equation*}
\{\lambda \in \mathfrak{L}\mid \phi (\lambda )\geq n\}=B_{n}\in \mathcal{U}.
\end{equation*}
\end{remark}

When the inclusion $\mathbb{R\subseteq K}$ is proper we have that $\mathbb{K}
$ is a superreal non Archimedean field\footnote{%
A superreal non Archimedean field is an ordered field that properly contains 
$\mathbb{R}$.}. In this case, it will be called a \textbf{hyperreal field}.
The terminology will be motivated by Cor \ref{beppino}, where we precise the
relationship (as fields) between the hyperreal field $\mathbb{K}$ and the
ultrapower $\mathbb{R}_{\mathcal{U}}^{\mathfrak{L}}$. Let us recall the
definition of $\mathbb{R}_{\mathcal{U}}^{\mathfrak{L}}:$

\begin{definition}
Let $\equiv _{\mathcal{U}}$ be the equivalence relation on $\mathbb{R}^{%
\mathfrak{L}}$ defined as follows: for every $\varphi ,\psi :\mathfrak{L}%
\rightarrow \mathbb{R}$%
\begin{equation*}
\varphi \equiv _{\mathcal{U}}\psi \Leftrightarrow \{\lambda \in \mathfrak{L}%
\mid \varphi (\lambda )=\psi (\lambda )\}\in \mathcal{U}.
\end{equation*}%
The equivalence class of every function $\varphi:\mathfrak{L}\rightarrow%
\mathbb{R}$ will be denoted by $[\varphi]_{\mathcal{U}}$. The ultrapower $%
\mathbb{R}_{\mathcal{U}}^{\mathfrak{L}}$ is the quotient $\mathbb{R}^{%
\mathfrak{L}}/\equiv _{\mathcal{U}}$.

The operations on $\mathbb{R}_{\mathcal{U}}^{\mathfrak{L}}$ are defined
componentwise: for every $\varphi ,\psi :\mathfrak{L}\rightarrow \mathbb{R}$
we set

\begin{equation*}
[\varphi]_{\mathcal{U}}+[\psi]_{\mathcal{U}}:=[\varphi+\psi]_{\mathcal{U}};
\ [\varphi]_{\mathcal{U}}+[\psi]_{\mathcal{U}}:=[\varphi\cdot\psi]_{\mathcal{%
U}}.
\end{equation*}
\end{definition}

A well-known result (see e.g. \cite{keisler}) is that $(\mathbb{R}^{%
\mathfrak{L}}_{\mathcal{U}},[0]_{\mathcal{U}},[1]_{\mathcal{U}},+,\cdot)$ is
a field. Moreover, we have the following:

\begin{corollary}
\label{beppino}$\mathbb{K}$ and$\mathbb{\ R}_{\mathcal{U}}^{\mathfrak{L}}$
are isomorphic as fields.
\end{corollary}

\begin{proof} The isomorphism is given by the map $\Psi :\mathbb{%
K\rightarrow R}_{\mathcal{U}}^{\mathfrak{L}}$ such that, for every $\varphi :%
\mathfrak{L\rightarrow }\mathbb{R},$ 
\begin{equation*}
\Psi \left( \lim_{\lambda \rightarrow \Lambda }\left( \lambda ,\varphi
(\lambda )\right) \right) =[\varphi ]_{\mathcal{U}}.
\end{equation*}

Condition (\ref{sei}) in the definition of $(\mathfrak{L},\mathcal{U})$%
-completion entails that $\Psi $ is injective, whereas the definition of $%
\mathbb{K}$ as the set of all possible $\Lambda $-limits entails that $\Psi $
is surjective. Since it is immediate to see that $\Psi $ also preserves the
operations, we have that it is an isomorphism. \end{proof}

We will strenghten Cor \ref{beppino} in Thm \ref{keisler}. By Cor \ref%
{beppino} it clearly follows that, if the $(\mathfrak{L},\mathcal{U})$%
-completion is minimal, as sets $\mathbb{R}_{\mathfrak{L}}\cong \left( 
\mathfrak{L}\times \mathbb{R}\right) \uplus \mathbb{R}_{\mathcal{U}}^{%
\mathfrak{L}}.$

\begin{remark}
\label{ignazio}Let us note that $(\left( \mathfrak{L}\times \mathbb{R}%
\right) \uplus \mathbb{R}_{\mathcal{U}}^{\mathfrak{L}},\tau )$ is a $(%
\mathfrak{L},\mathcal{U})$-completion of $\mathbb{R}$ for different choices
of $\tau $. One such choice is the topology $\tau _{\mathcal{U}}$ introduced
in the proof of Theorem \ref{nuovo}; a different topology can be constructed
as follows: let us fix a function $\varphi $ with $\lim_{\lambda \rightarrow
\Lambda }\left( \lambda ,\varphi (\lambda )\right) \notin \mathbb{R}$, a
nonempty infinite set $B\notin \mathcal{U}$, a free filter $\mathcal{F}$ on $%
B$ and let us consider the following topology $\widetilde{\tau }$on $\left( 
\mathfrak{L}\times \mathbb{R}\right) \uplus \mathbb{R}_{\mathcal{U}}^{%
\mathfrak{L}}$: if $\xi \neq \lim_{\lambda \rightarrow \Lambda }\left(
\lambda ,\varphi (\lambda )\right) $ then a family of open neighborhoods of $%
\xi $ is 
\begin{equation*}
\left\{ O_{\psi ,Q}\mid Q\in \mathcal{U}\text{, }\psi \text{ function with }%
\xi =\lim_{\lambda \rightarrow \Lambda }\left( \lambda ,\psi (\lambda
)\right) \right\} ;
\end{equation*}%
if $\xi =\lim_{\lambda \rightarrow \Lambda }\left( \lambda ,\varphi (\lambda
)\right) $ then a family of open neighborhoods of $\xi $ is 
\begin{equation*}
\{O_{F,Q}\mid F\in \mathcal{F},Q\in \mathcal{U\}}
\end{equation*}%
where, for every $F\in \mathcal{F}$, $Q\in \mathcal{U}$ we set%
\begin{equation*}
O_{F,Q}=O_{\varphi ,Q}\cup \{(\lambda ,x)\mid \lambda \in F,x\in \mathbb{R}%
\}.
\end{equation*}%
By construction, $\left( \left( \mathfrak{L}\times \mathbb{R}\right) \uplus 
\mathbb{R}_{\mathcal{U}}^{\mathfrak{L}},\widetilde{\tau }\right) \ $is a $%
\left( \mathfrak{L},\mathcal{U}\right) $-completion of $\mathbb{R}$.
\end{remark}

A consequence of Remark \ref{ignazio} is that there are infinitely many
topologies $\tau $ that make $(\left( \mathfrak{L}\times \mathbb{R}\right)
\uplus \mathbb{R}_{\mathcal{U}}^{\mathfrak{L}},\tau )$ a $(\mathfrak{L},%
\mathcal{U})$-completion of $\mathbb{R}$. However, the topology introduced
in the proof of Thm \ref{nuovo} plays a central role in our approach. For
this reason, we introduce the following definition.

\begin{definition}
Let $\left( \mathbb{R}_{\mathfrak{L}},\tau \right) $ be a $(\mathfrak{L},%
\mathcal{U})$-completion of $\mathbb{R}$. We call \textbf{slim topology},
and we denote by $\tau _{\mathcal{U}},$ the topology on $\mathbb{R}_{%
\mathfrak{L}}$ generated by the family of open sets%
\begin{equation*}
\left\{ N_{\varphi ,Q}\ |\ \varphi \in \mathfrak{F}\left( \mathfrak{L},%
\mathbb{R}\right) ,Q\in \mathcal{U}\right\} \cup \mathcal{P}(\mathfrak{L}%
\times \mathbb{R})
\end{equation*}%
where, for every $\varphi \in \mathfrak{F}\left( \mathfrak{L},\mathbb{R}%
\right) ,~Q\in \mathcal{U}$ we set%
\begin{equation*}
N_{\varphi ,Q}:=\left\{ \left( \lambda ,\varphi (\lambda )\right) \mid
\lambda \in Q\right\} \cup \left\{ \lim_{\lambda \rightarrow \Lambda }\left(
\lambda ,\varphi (\lambda )\right) \right\} .
\end{equation*}
\end{definition}

\begin{proposition}
\label{gollum}The slim topology $\tau _{\mathcal{U}}$ is finer than any
topology $\tau $ that makes $\left( \left( \mathfrak{L}\times \mathbb{R}%
\right) \uplus \mathbb{R}_{\mathcal{U}}^{\mathfrak{L}},\tau \right) $ a $%
\left( \mathfrak{L},\mathcal{U}\right) $-completion of $\mathbb{R}$.
\end{proposition}

\begin{proof} Let $\tau $ be given, let $O$ be an open set in $\tau $ and
let $x\in O$. If $x\in \mathfrak{L}\times \mathbb{R}$ then $\{x\}$ is an
open neighborhood of $x$ in $\tau _{\mathcal{U}}$ contained in $O$; if $%
x=\lim_{\lambda \rightarrow \Lambda }\left( \lambda ,\varphi (\lambda
)\right) $ for some function $\varphi :\mathfrak{L\rightarrow }\mathbb{R}$
then let $B\in \mathcal{U}$ be such that $\{(\lambda ,\varphi (\lambda
))\mid \lambda \in B\}\subseteq O;$ therefore, by construction, $O_{\varphi
,B}$ is an open neighborhood of $x$ in $\tau _{\mathcal{U}}$ entirely
contained in $O$. This proves that $O$ is an open set in $\tau _{\mathcal{U}}
$, therefore $\tau _{\mathcal{U}}$ is finer than $\tau .$ \end{proof}

The slim topology can also be characterized in terms of closure of subsets
of $\left( \mathfrak{L}\times \mathbb{R}\right) :$

\begin{proposition}
Let $\left( \left( \mathfrak{L}\times \mathbb{R}\right) \uplus \mathbb{R}_{%
\mathcal{U}}^{\mathfrak{L}},\tau \right) $ be a $\left( \mathfrak{L},%
\mathcal{U}\right) $-completion of $\mathbb{R}$. The following facts are
equivalent:

\begin{enumerate}
\item $\tau =\tau _{\mathcal{U}};$

\item for every set $B\subseteq \left( \mathfrak{L}\times \mathbb{R}\right) $
we have that 
\begin{equation*}
cl_{\tau }(B)=B\cup \left\{ \lim_{\lambda \rightarrow \Lambda }\left(
\lambda ,\varphi (\lambda )\right) \mid \exists A\in \mathcal{U~}\forall
\lambda \in A~(\lambda ,\varphi (\lambda ))\in B\right\} .
\end{equation*}
\end{enumerate}
\end{proposition}

\begin{proof} $(1)\Rightarrow (2)$ Let $\varphi:\mathfrak{L}\rightarrow\mathbb{R}$, let $B\subseteq (\mathfrak{L}\times\mathbb{R})$ and let $\xi=\lim\limits_{\lambda\rightarrow\Lambda}(\lambda,\varphi(\lambda))$. Let $%
A=\{\lambda \in \mathfrak{L}\mid (\lambda ,\varphi (\lambda ))\in B\}.$ If $%
A\in \mathcal{U}$ then for every open neighborhood $O$ of $\xi$ we have that $O\cap B\neq \emptyset $ by construction, so $\xi
\in cl_{\tau _{\mathcal{U}}}(B);$ if $A\notin \mathcal{U}$ then $O_{\varphi
,A}$ is a neighborhood of $\xi $ such that $O_{\varphi ,A}\cap B=\emptyset ,$
therefore $\xi \notin cl_{\tau _{\mathcal{U}}}(B).$

$(2)\Rightarrow (1)$ Let $A\in \mathcal{U}$, let $\varphi :\mathfrak{L}%
\rightarrow \mathbb{R}$ and let $\xi =\lim_{\lambda \rightarrow \Lambda
}\left( \lambda ,\varphi (\lambda )\right) .$ Let us consider $B=\left( 
\mathfrak{L}\times \mathbb{R}\right) \setminus O_{A,\varphi }.$ By
hypothesis and construction 
\begin{equation*}
cl_{\tau }(B)=\left[ \left( \mathfrak{L}\times \mathbb{R}\right) \uplus 
\mathbb{R}_{\mathcal{U}}^{\mathfrak{L}}\right] \setminus O_{A,\varphi }.
\end{equation*}%
Therefore $O_{A,\varphi }$ is open for every $A\in \mathcal{U},~\varphi :%
\mathfrak{L}\rightarrow \mathbb{R}$, so $\tau $ is finer than $\tau _{%
\mathcal{U}}$ which, as a consequence of Proposition \ref{gollum}, entails
that $\tau =\tau _{\mathcal{U}}.$ \end{proof}

\begin{definition}
We will call $\left( \left( \mathfrak{L}\times \mathbb{R}\right) \uplus 
\mathbb{R}_{\mathcal{U}}^{\mathfrak{L}},\tau _{\mathcal{U}}\right) $ the 
\textbf{canonical} $\left( \mathfrak{L},\mathcal{U}\right) $-completion of $%
\mathbb{R}$.
\end{definition}

From the next section on we will work only with the minimal canonical $%
\left( \mathfrak{L},\mathcal{U}\right) $-completion of $\mathbb{R}$.

\subsection{Natural extension of sets and functions}

From now on, $\overline{\left( \cdot \right) }$ will denote the closure
operator in the canonical $\left( \mathfrak{L},\mathcal{U}\right) $%
-completion of $\mathbb{R}.$

\begin{definition}
For every $E\subseteq \mathbb{R}$ we set%
\begin{equation*}
E_{\mathfrak{L}}:=\overline{\mathfrak{L}\times E}.
\end{equation*}
\end{definition}

A different and related (as we will show in Prop \ref{sandokan}) extension
of $E$ is the following:

\begin{definition}
\label{janez}Given \textit{a set }$E\subset \mathbb{R}$, we set%
\begin{equation*}
E^{\ast }:=\left\{ \lim_{\lambda \rightarrow \Lambda }\left( \lambda ,\psi
(\lambda )\right) |\psi (\lambda )\in E\right\} ;
\end{equation*}%
$E^{\ast }$ is called the \textbf{natural extension }of $E.$
\end{definition}

Let us observe that by property (\ref{quattro}) of the definition of $\left( 
\mathfrak{L},\mathcal{U}\right) $-completions it follows that $E\subseteq
E^{\ast }.$ Following the notation introduced in Def \ref{janez}, from now
on we will denote $\mathbb{K}$ by $\mathbb{R}^{\ast }.$

It is easy to modify the proof of Prop \ref{carattere} to obtain the
following result:

\begin{proposition}
\label{sandokan}For every $E\subseteq \mathbb{R}$ we have that $E_{\mathfrak{%
L}}=\left( \mathfrak{L}\times E\right) \uplus E^{\ast }$.
\end{proposition}

It is also possible to extend functions to $\mathbb{R}_{\mathfrak{L}}.$ To
this aim, given a function

\begin{equation*}
f:A\rightarrow B
\end{equation*}%
we will denote by%
\begin{equation*}
f_{\mathfrak{L}}:\mathfrak{L}\times A\rightarrow \mathfrak{L}\times B
\end{equation*}%
the function defined as follows:%
\begin{equation*}
f_{\mathfrak{L}}\left( \lambda ,x\right) =\left( \lambda ,f(x)\right) .
\end{equation*}

\begin{lemma}
\label{agnese}For every $A,B\subseteq \mathbb{R}$, for every function $%
f:A\rightarrow B,$ $f$ can be extended to a continuous function%
\begin{equation*}
\overline{f_{\mathfrak{L}}}:A_{\mathfrak{L}}\rightarrow B_{\mathfrak{L}}.
\end{equation*}%
Moreover, the restriction of $\overline{f_{\mathfrak{L}}}$ to $A$ coincides
with $f.$
\end{lemma}

\begin{proof} The extension of $f$ to $\mathfrak{L}\times A$ is given by $f_{\mathfrak{L}}$. Therefore to get the desired extension to $A_{\mathfrak{L}}$ it is sufficient to extend $f_{\mathfrak{L}}$ on 
$A^{\ast }.$ For every $\varphi \in A^{\mathfrak{L}}$ we set 
\begin{equation*}
\overline{f_{\mathfrak{L}}}\left( \lim_{\lambda \rightarrow \Lambda }\left(
\lambda ,\varphi (\lambda )\right) \right) =\lim_{\lambda \rightarrow
\Lambda }\left( \lambda ,f\left( \varphi (\lambda )\right) \right) .
\end{equation*}%
Let us note that the definition is well posed and that $\overline{f_{%
\mathfrak{L}}}\left( \lim_{\lambda \rightarrow \Lambda }\left( \lambda
,\varphi (\lambda )\right) \right) \in B^{\ast }$ since, for every $\varphi
\in A^{\mathfrak{L}}$, the function $f\circ \varphi \in B^{\mathfrak{L}}$.
This extension is continuous: let $\Omega $ be a basis open subset of $B_{%
\mathfrak{L}}$. If $\Omega =\{(\lambda ,x)\}$ then 
\begin{equation*}
\overline{f_{\mathfrak{L}}}^{-1}(\Omega )=\bigcup_{y\in f^{-1}(x)}(\lambda
,y),
\end{equation*}%
which is open. If $\Omega =N_{\varphi ,Q}$ for some $\varphi :\mathfrak{L}%
\rightarrow \mathbb{R}$, $Q\in \mathcal{U}$ then let $\xi \in \overline{f_{%
\mathfrak{L}}}^{-1}(\Omega ).$ If $\xi =(\lambda ,x)$ for some $x\in A$ then 
$\{(\lambda ,x)\}$ is a neighborhood of $(\lambda ,x)$ included in $%
\overline{f_{\mathfrak{L}}}^{-1}(\Omega );$ if $\xi =\lim_{\lambda \rightarrow \Lambda }(\lambda,\psi(\lambda))$ then $%
\overline{f_{\mathfrak{L}}}(\xi)=\lim_{\lambda \rightarrow \Lambda }(\lambda,\varphi(\lambda)),$ therefore there exists $%
Q_{1}\in \mathcal{U}$ such that $f(\psi (\lambda ))=\varphi (\lambda )$ for
all $\lambda \in Q_{1}$, hence if we set $Q_{2}=Q\cap Q_{1}$ we have that $%
N_{\psi ,Q_{2}}$ is a neighborhood of $\xi $ included in $\overline{f_{%
\mathfrak{L}}}^{-1}(\Omega ),$ thus $\overline{f_{\mathfrak{L}}}^{-1}(\Omega
)$ is open, and this proves that $\overline{f_{\mathfrak{L}}}$ is continuous.

Finally, $\overline{f_{\mathfrak{L}}}$ restricted to $A$ coincides with $f$
since, for every $a\in A,$ by definition%
\begin{equation*}
\overline{f_{\mathfrak{L}}}(a)=\overline{f_{\mathfrak{L}}}\left(
\lim_{\lambda \rightarrow \Lambda }\left( \lambda ,a\right) \right)
=\lim_{\lambda \rightarrow \Lambda }\left( \lambda ,f\left( a\right) \right)
=f(a).\qedhere
\end{equation*}

\end{proof}

Lemma \ref{agnese} entails that the following definition is well posed:

\begin{definition}
Given a function%
\begin{equation*}
f:A\rightarrow B
\end{equation*}%
the restriction of $\overline{f_{\mathfrak{L}}}$ to $A^{\ast }$ is called
the \textbf{natural extension} of $f$ and it will be denoted by%
\begin{equation*}
f^{\ast }:A^{\ast }\rightarrow B^{\ast }.
\end{equation*}
\end{definition}

In particular, $f^{\ast }(a)=f(a)$ for every $a\in A.$

\subsection{The $\Lambda $-limit\ in\ $V_{\infty }(\mathbb{R)}$\label{marina}%
}

In this section we want to extend the notion of $\Lambda $-limit to a wider
family of functions. To do that, we have to introduce the notion of
superstructure on a set (see also \cite{keisler}):

\begin{definition}
Let $E$ be an infinite set. The superstructure on $E$ is the set 
\begin{equation*}
V_{\infty }(E)=\bigcup_{n\in \mathbb{N}}V_{n}(E),
\end{equation*}%
where the sets $V_{n}(E)$ are defined by induction by setting%
\begin{equation*}
V_{0}(E)=E
\end{equation*}%
and, for every $n\in \mathbb{N}$, 
\begin{equation*}
V_{n+1}(E)=V_{n}(E)\cup \mathcal{P}\left( V_{n}(E)\right) .
\end{equation*}
\end{definition}

Here $\mathcal{P}\left( E\right) $ denotes the power set of $E.$ Identifying
the couples with the Kuratowski pairs and the functions and the relations
with their graphs, it follows that{\ }$V_{\infty }(E\mathbb{)}$ contains
almost every usual mathematical object that can be constructed starting with 
$E;$ in particular, $V_{\infty }(\mathbb{R)}$ contains almost every usual
mathematical object of analysis.

Sometimes, following e.g. \cite{keisler}, we will refer to%
\begin{equation*}
\mathbb{U}:=V_{\infty }(\mathbb{R)}
\end{equation*}%
as to the \textbf{standard universe}. A mathematical entity (number, set,
function or relation) is said to be \textbf{standard} if it belongs to $%
\mathbb{U}$.

Now we want to formally define the $\Lambda $-limit of $(\lambda ,\varphi
(\lambda ))$ where $\varphi (\lambda )$ is any bounded function of
mathematical objects in $V_{\infty }(\mathbb{R)}$ (a function $\varphi :%
\mathfrak{L}\rightarrow V_{\infty }(\mathbb{R)}$ is called bounded if there
exists $n$ such that $\forall \lambda \in \mathfrak{L},\ \varphi (\lambda
)\in V_{n}(\mathbb{R)}$). To this aim, let us consider a function%
\begin{equation}
\varphi :\mathfrak{L}\rightarrow V_{n}(\mathbb{R)}.  \label{net}
\end{equation}%
We will define $\lim_{\lambda \rightarrow \Lambda }(\lambda ,\varphi
(\lambda ))$ by induction on $n$.

\begin{definition}
\label{def}For $n=0,$ $\lim_{\lambda \rightarrow \Lambda }(\lambda ,\varphi
(\lambda ))$ exists by Thm \ref{nuovo}; so by induction we may assume that
the limit is defined for $n-1$ and we define it for the function (\ref{net})
as follows:%
\begin{equation*}
\lim_{\lambda \rightarrow \Lambda }(\lambda ,\varphi (\lambda ))=\left\{
\lim_{\lambda \rightarrow \Lambda }(\lambda ,\psi (\lambda ))\ |\ \psi :%
\mathfrak{L}\rightarrow V_{n-1}(\mathbb{R)}\text{ and}\ \forall \lambda \in 
\mathfrak{L},\ \psi (\lambda )\in \varphi (\lambda )\right\} .
\end{equation*}
\end{definition}

Clearly $\lim_{\lambda \rightarrow \Lambda }(\lambda ,\varphi (\lambda ))$
is a well defined set in $V_{\infty }(\mathbb{R}^{\ast }).$

\begin{definition}
A mathematical entity (number, set, function or relation) which is the $%
\Lambda $-limit of a function is called \textbf{internal}.
\end{definition}

Notice that $V_{\infty }(\mathbb{R}^{\ast })$ contains sets which are not
internal.

\begin{example}
Each real number is standard and internal. However the set of real numbers $%
\mathbb{R}\in V_{\infty }(\mathbb{R}^{\ast })$ is standard, but not
internal. In order to see this let us suppose that there is a function $%
\varphi :\mathfrak{L}\rightarrow V_{1}(\mathbb{R)}$ such that $\mathbb{R=}%
\lim_{\lambda \rightarrow \Lambda }(\lambda ,\varphi (\lambda )).$
Therefore, by definition, we would have 
\begin{equation*}
\mathbb{R=}\left\{ \lim_{\lambda \rightarrow \Lambda }(\lambda ,\psi
(\lambda ))\ |\ \psi :\mathfrak{L}\rightarrow \mathbb{R}\text{ and}\ \forall
\lambda \in \mathfrak{L},\ \psi (\lambda )\in \varphi (\lambda )\right\} .
\end{equation*}%
In particular, for every constant $c\in \mathbb{R}$ we have that $c\in
\varphi (\lambda )$; therefore, $\varphi (\lambda )=\mathbb{R}$ for every $%
\lambda \in \mathfrak{L}$, and this is absurd because 
\begin{equation*}
\lim_{\lambda \rightarrow \Lambda }\left( \lambda ,\mathbb{R}\right) =%
\mathbb{R}^{\ast },
\end{equation*}%
and (except trivial cases) $\mathbb{R}^{\ast }$ properly includes $\mathbb{R}
$. Let us explicitly observe that (except trivial cases), while for every $%
c\in \mathbb{R}$ the function $\lambda \rightarrow (\lambda ,c)$ converges
to $c$, given $A\in V_{n}(\mathbb{R)}$, for $n\geq 1$ the function $\lambda
\rightarrow (\lambda ,A)$ converges to a proper superset of $A$.
\end{example}

\begin{definition}
A mathematical entity (number, set, function or relation) which is not
internal is called \textbf{external.}
\end{definition}

As it is given, the definition of limit given by Def \ref{def} is not
related to any topology. Thus a question arises naturally: is there a
topological Hausdorff space such that the limit given by Def \ref{def} is
the topological limit of a function?

The answer is affirmative, and it is a consequence of the possibility to
topologize the set 
\begin{equation*}
\mathbb{U}_{\mathfrak{L}}=\left[ \mathfrak{L}\times V_{\infty }(\mathbb{R})%
\right] \uplus V_{\infty }(\mathbb{R}^{\ast }).
\end{equation*}%
To topologize $\mathbb{U}_{\mathfrak{L}}$ we take as open sets:

\begin{itemize}
\item every subset of $\mathfrak{L}\times V_{\infty }(\mathbb{R})$;

\item $\left\{ x\right\} $ for every $x\in V_{\infty }(\mathbb{R}^{\ast })$
that is external;

\item $N_{\varphi ,Q}:=\left\{ \left( \lambda ,\varphi (\lambda )\right)
|\lambda \in Q\right\} \cup \left\{ x\right\}$ for every $x$ internal such
that $\varphi $ is a bounded sequence with 
\begin{equation*}
x=\lim_{\lambda \rightarrow \Lambda }\left( \lambda ,\varphi (\lambda
)\right).
\end{equation*}
\end{itemize}

We let $\sigma _{\mathcal{U}}$ be the topology on $\mathbb{U}_{\mathfrak{L}}$
generated by these open sets.t is clear that this topology is Hausdorff and
that the $\Lambda $-limit is a limit in this topology.

The set%
\begin{equation*}
\mathbb{U}_{\mathfrak{L}}=\left[ \mathfrak{L}\times V_{\infty }(\mathbb{R})%
\right] \cup V_{\infty }(\mathbb{R}^{\ast })
\end{equation*}%
will be called the \textbf{expanded universe}. Let us note that, by
construction, $\mathbb{U}_{\mathfrak{L}}\subseteq V_{\infty }(\mathbb{R}_{%
\mathfrak{L}}).$

The results about extensions of subsets of $\mathbb{R}$ and of functions $%
f:A\rightarrow B$, $A,B\subseteq \mathbb{R}$, can be generalized to our new
general setting. Since a function $f$ can be identified with its graph then
the natural extension of a function is defined by the above definition.
Moreover we have the following result, that can be proved as Lemma \ref%
{agnese}:

\begin{theorem}
For every sets $E,F\in V_{\infty }(\mathbb{R})$ and for every function $%
f:E\rightarrow F$ the \textbf{natural extension} of $f$ is a continuous
function 
\begin{equation*}
f^{\ast }:E^{\ast }\rightarrow F^{\ast },
\end{equation*}%
and for every function $\varphi :\mathfrak{L}\rightarrow E$ we have that%
\begin{equation*}
\lim_{\lambda \rightarrow \Lambda }\ f(\lambda ,\varphi (\lambda ))=f^{\ast
}\left( \lim_{\lambda \rightarrow \Lambda }(\lambda ,\varphi (\lambda
))\right) .
\end{equation*}
\end{theorem}

\section{Comparison between $\Lambda $-theory and ultrapowers\label%
{comparison}}

\subsection{$\Lambda $-theory and nonstandard universes}

It should be evident to any reader with a background in NSA that $\Lambda$%
-theory (when restricted to minimal canonical extensions) is closely related
to ultrapowers (which, from a purely logical point of view, are even easier
to define). In this section we want to detail the relationship between $%
\Lambda $-theory and NSA. We will show that $\mathbb{U}_{\mathfrak{L}}$
contains a nostandard universe in the sense of Keisler \cite{keisler}. We
recall the main definitions of \cite{keisler}.

\begin{definition}
\label{SE}A \textbf{superstructure embedding} is a one to one mapping $\ast $
of $V_{\infty }(\mathbb{R)}$ into another superstructure $V_{\infty }(%
\mathbb{S)}$ such that

\begin{enumerate}
\item $\mathbb{R}$ is a proper subset of $\mathbb{S}$, $r^{\ast }=r$ for all 
$r\in \mathbb{R}$, and $\mathbb{R}^{\ast }=\mathbb{S}$;

\item for $x,y\in V_{\infty }(\mathbb{R)}$, $x\in y$ if and only if $x^{\ast
}\in y^{\ast }$.
\end{enumerate}
\end{definition}

To avoid confusion, in this section we will use the letter $\mathbb{K}$ to
denote the non-Archimedean field constructed in Section \ref{HF}, while $%
\mathbb{R}^{\ast }$ will be used as in Def \ref{SE}.

Let us denote by $\mathcal{L}$ a formal language relative to a first order
predicate logic with the equality symbol, a binary relation symbol $\in $,
and a constant symbol for each element in $V_{\infty }(\mathbb{R)}$. We
recall that a sentence $p\in \mathcal{L}$ is bounded if every quantifier in $%
p$ is bounded (see e.g. \cite{keisler}). The notion of bounded sequence
allows to define the notion of nonstandard universe.

\begin{definition}
\label{truzzo}A \textbf{nonstandard universe} is a superstructure embedding $%
\ast :V_{\infty }(\mathbb{R)}\rightarrow V_{\infty }(\mathbb{R}^{\mathbb{%
\ast }}\mathbb{)}$ which satisfies Leibniz' Principle, which is the property
that states that for each bounded sentence $p\in \mathcal{L},$ $p$ is true
in $V_{\infty }(\mathbb{R)}$ if and only if $p^{\ast }$ is true\footnote{$%
p^{\ast }$ is the bounded sentence obtained by changing every constant
symbol $c\in V_{\infty }(\mathbb{R})$ that appears in $p$ with $c^{\ast }.$}
in $V_{\infty }(\mathbb{R}^{\mathbb{\ast }}\mathbb{)}$.
\end{definition}

\begin{definition}
\label{star}We let $\ast :V_{\infty }(\mathbb{R)}\rightarrow V_{\infty }(%
\mathbb{K)}$ be the map defined as follows: for every element $x\in
V_{\infty }(\mathbb{R)}$ we set%
\begin{equation*}
x^{\ast }=\lim_{\lambda \rightarrow \Lambda }(\lambda ,x).
\end{equation*}
\end{definition}

\begin{remark}
Following Keisler (see \cite{keisler}), in Def \ref{truzzo} we have called
nonstandard universe just the superstructure embedding; however, in our
approach, probably, it would be more appropriate to call nonstandard
universe the set $V_{\infty }(\mathbb{K)};$ in this case the global picture
would be the following one: the extended universe%
\begin{equation*}
\mathbb{U}_{\mathfrak{L}}=\left[ \mathfrak{L}\times V_{\infty }(\mathbb{R})%
\right] \uplus V_{\infty }(\mathbb{K})
\end{equation*}%
contains pairs $(\lambda ,x)$ and elements of the nostandard universe $%
V_{\infty }(\mathbb{K});$ the latter contains the following objects:

\begin{itemize}
\item standard elements, namely objects $x\in V_{\infty }(\mathbb{R})\subset
V_{\infty }(\mathbb{K)}$;

\item nonstandard elements, namely objects $x\in V_{\infty }(\mathbb{K)}%
\backslash V_{\infty }(\mathbb{R})$;

\item hyperimages, namely objects $x$ such that there exists $y\in V_{\infty
}(\mathbb{R})$ with $x=y^{\ast }$;

\item internal objects, namely $\Lambda $-limits of bounded functions;

\item external objects.
\end{itemize}
\end{remark}

To give some examples: 7, $\mathbb{R}$, $\mathcal{P}(\mathbb{R}\times 
\mathcal{P}(\mathbb{R}))$ are all standard elements; 7 is also an
hyperimage, while $\mathbb{R}$, $\mathcal{P}(\mathbb{R}\times \mathcal{P}(%
\mathbb{R}))$ are not; $\mathbb{K},$ $\mathcal{P}(\mathbb{R)}^{\ast }$ and $%
\lim_{\lambda \rightarrow \Lambda }(\lambda ,\varphi (\lambda ))$ for every $%
\varphi :\mathfrak{L}\rightarrow \mathbb{R}$ which is not eventually
constant are nonstandard elements, and they are all internal; $\mathbb{R}$
and $\mathbb{K\setminus R}$ are external objects.

An interesting class of internal objects, particularly important for our
applications to PDEs, is that of hyperfinite objects\footnote{%
See e.g. \cite{Albe}, where many different applications of hyperfinite
objects and other nonstandard tools are developed}:

\begin{definition}
An object $\xi \in V_{\infty }(\mathbb{K)}$ is hyperfinite if there exists a
natural number $n$ and a bounded function $\varphi :\mathfrak{L}\rightarrow 
\mathcal{P}_{fin}(V_{n}(\mathbb{R))}$ such that $\xi =\lim_{\lambda
\rightarrow \Lambda }(\lambda ,\varphi (\lambda )).$
\end{definition}

Hyperfinite objects are the analogue, in the universe $V_{\infty }(\mathbb{K}%
),$ of finite objects in $V_{\infty }(\mathbb{R}).$ The notion of
hyperfinite object will be used in Section \ref{GS} to show some
applications of $\Lambda $-theory.

To detail the relationship between $\Lambda $-theory and nonstandard
universes in the sense of Keisler we need to specify how we interpret
formulas in $V_{\infty }(\mathbb{K})$\footnote{%
Once again, it should be evident to readers expert in NSA that our
definition is precisely analogue to the one that is given for ultrapowers.}:

\begin{definition}
\label{interpretazione}Let $p(x_{1},\dots ,x_{n})\in \mathcal{L}$ be a
bounded formula having $x_{1},\dots ,x_{n}$ as its only free variables. Let $%
\xi _{1}=\lim_{\lambda \rightarrow \Lambda }(\lambda ,\varphi _{1}(\lambda
)),\dots ,\xi _{n}=\lim_{\lambda \rightarrow \Lambda }(\lambda ,\varphi
_{n}(\lambda )).$ We say that $p^{\ast }(\xi _{1},\dots ,\xi _{n})$ holds in 
$V_{\infty }(\mathbb{K})$ iff $p(\varphi _{1}(\lambda ),\dots ,\varphi
_{n}(\lambda ))$ is eventually true in $V_{\infty }(\mathbb{R})$, namely iff%
\begin{equation*}
\{(\lambda,(\varphi_{1}(\lambda),\dots,\varphi_{n}(\lambda))\mid p(\varphi
_{1}(\lambda ),\dots ,\varphi _{n}(\lambda )) \ \text{holds in} \ V_{\infty
}(\mathbb{R})\}\cup \{(\xi_{1},\dots,\xi_{n})\}
\end{equation*}
is open in $\sigma_{\mathcal{U}}$.
\end{definition}

\begin{theorem}
\label{keisler}Let $\ast $ be defined as in Def \ref{star}; then 
\begin{equation*}
\left( V_{\infty }(\mathbb{R)},V_{\infty }(\mathbb{K)},\ast \right)
\end{equation*}%
is a nonstandard universe.
\end{theorem}

\begin{proof} That $\ast :V_{\infty }(\mathbb{R)}\rightarrow V_{\infty }(%
\mathbb{K)}$ is a superstructure embedding follows clearly from the
definitions. \par
Moreover, for every bounded formula $p(x_{1},\dots ,x_{n})\in 
\mathcal{L}$ having $x_{1},\dots ,x_{n}$ as its only free variables, for
every $\xi _{1}=\lim_{\lambda \rightarrow \Lambda }(\lambda ,\varphi
_{1}(\lambda )),\dots ,\xi _{n}=\lim_{\lambda \rightarrow \Lambda }(\lambda
,\varphi _{n}(\lambda )),$ we have that%
\begin{equation*}
\begin{array}{c}
p(\xi _{1},\dots ,\xi _{n})\text{ holds in }V_{\infty }(\mathbb{K}%
)\Leftrightarrow \\ 
\{(\lambda,(\varphi_{1}(\lambda),\dots,\varphi_{n}(\lambda))\mid p(\varphi _{1}(\lambda ),\dots ,\varphi
_{n}(\lambda ))\text{ holds in }V_{\infty }(\mathbb{R})\}\cup \{(\xi_{1},\dots,\xi_{n})\}\\ \ \text{is open in} \ \sigma_{\mathcal{U}}
\Leftrightarrow \\
\{\lambda \in \mathcal{L}\mid p(\varphi _{1}(\lambda ),\dots ,\varphi
_{n}(\lambda ))\text{ holds in }V_{\infty }(\mathbb{R})\}\in \mathcal{U}%
\Leftrightarrow \\ 
p([\varphi _{1}],\dots ,[\varphi _{n}])\text{ holds in }\mathbb{R}_{\mathcal{%
U}}^{\mathfrak{L}}.%
\end{array}%
\end{equation*}%
This equivalence can be used to easily prove the transfer property for $\ast
:V_{\infty }(\mathbb{R)}\rightarrow V_{\infty }(\mathbb{K)}$ by induction on
the complexity of formulas. \end{proof}

\subsection{General remarks}

Theorem \ref{keisler} precises the intuition that the topological approach
to non-Archimedean mathematics given by $\Lambda $-theory is closely related
with NSA as presented by Keisler in \cite{keisler}. As we said in the
introduction, we think to $\Lambda$-theory as a way to present to a
non-expert reader many basic ideas of NSA in a more familiar language.
Nevertheless, we think that from a philosophical point of view point there
are some differences between $\Lambda$-theory and the ultrapowers approach:

\begin{enumerate}
\item in $\Lambda $-theory we assume the existence of a unique mathematical
universe $\mathbb{U}_{\mathfrak{L}}\subset V_{\infty }(\mathfrak{L}\cup 
\mathbb{K)}$. Inside this universe there are entities that do not appear in
traditional mathematics but that can be obtained as limits of traditional
objects, namely the internal elements. Moreover, there are also external
objects, and some of them are objects of traditional mathematics (e.g., $%
\mathbb{R}$);

\item in NSA the primitive concept is that of hyperimage, the other concepts
(e.g., the concept of internal object) are derived by that one; in $\Lambda $%
-theory, the primitive concept is that of $\Lambda $-limit, while the
concept of hyperimage is derived by the limit. So, within $\Lambda $-theory
the notion of internal object (being defined as a $\Lambda $-limit) is more
primitive than that of hyperimage;

\item the construction of the hyperreal field in our approach has a
topological "flavour" which is similar to other constructions in traditional
mathematics. In fact, e.g. whitin our approach the construction of $\mathbb{R%
}^{\ast }$ as "set of limits of functions with values in $\mathfrak{L}\times 
\mathbb{R}$" has some similarities with the construction of $\mathbb{R}$ as
set of limits of Cauchy sequences with values in $\mathbb{Q}$.
\end{enumerate}

\section{Generalized Solutions\label{GS}}

In many circumstances, the notion of function is not sufficient to the needs
of a theory and it is necessary to extend it. Many different constructions
have been considered in the literature to deal with this problem, both with
standard (for example, Colombeau's Theory, see e.g. \cite{Gro} and
references therein for a complete presentation of the theory and \cite{GioLu}
and reference therein for some new developments of the theory with
applications to generalized ODE's) and nonstandard techniques (see e.g. \cite%
{Rob2}). In this section we want to apply $\Lambda $-theory to construct
spaces of generalized functions called ultrafunctions (see also \cite{ultra}%
, \cite{belu2012}, \cite{belu2013}, \cite{milano}, \cite{algebra}, \cite%
{beyond}, \cite{gauss}), and to use them to study a simple class of problems
in calculus of variations. As we are going to show, ultrafunctions are
constructed by means of a particular version of the hyperfinite approach
which can be naturally introduced by means of $\Lambda $-theory.

In this section we will use the following shorthand notation: for every
bounded function $\varphi :\mathfrak{L}\rightarrow V_{\infty }(\mathbb{R})$
we let 
\begin{equation*}
\lim_{\lambda \uparrow \Lambda }\varphi (\lambda ):=\lim_{\lambda
\rightarrow \Lambda }\left( \lambda ,\varphi (\lambda )\right) .
\end{equation*}

\subsection{Ultrafunctions}

Let $N$ be a natural number, let $\Omega $ be a set in $\mathbb{R}^{N}$ and
let $V(\Omega )$ be a function vector space. We want to define the space of
ultrafunctions generated by $V(\Omega ).$ We assume that%
\begin{equation*}
\mathfrak{L}=\mathcal{P}_{fin}\left( V(\Omega)\right),
\end{equation*}

and we let $\mathcal{U}$ be a fine ultrafilter\footnote{%
Let us recall that an ultrafilter $\mathcal{U}$ on $\mathfrak{L}$ is fine if
for every $\lambda\in\mathfrak{L}$ the set $\{\mu\in\mathfrak{L}\mid
\mu\subseteq\lambda\}\in\mathcal{U}$. We also point out that, for more
complicated applications, it would be better to take $\mathfrak{L}=\mathcal{P%
}_{fin}\left( V_{\infty }(\mathbb{R})\right)$.} on $\mathfrak{L}$. For any $%
\lambda \in \mathfrak{L},$ we set%
\begin{equation*}
V_{\lambda }(\Omega )=Span\left\{ \lambda \cap V(\Omega )\right\} .
\end{equation*}

Let us note that, by construction, $V_{\lambda }(\Omega )$ is a finite
dimensional vector subspace of $V(\Omega )$.

\begin{definition}
Given the function space $V(\Omega )$ we set%
\begin{equation*}
V_{\Lambda }(\Omega ):=\lim_{\lambda \uparrow \Lambda }V_{\lambda }(\Omega
)=\left\{ \lim_{\lambda \uparrow \Lambda }u_{\lambda }~|~u_{\lambda }\in
V_{\lambda }(\Omega )\right\} .
\end{equation*}%
$V_{\Lambda }(\Omega )$ will be called the \textbf{space of ultrafunctions}
generated by $V(\Omega ).$
\end{definition}

Given any vector space of functions $V(\Omega )$, we have the following
three properties:

\begin{enumerate}
\item the ultrafunctions in $V_{\Lambda }(\Omega )$ are $\Lambda $-limits of
functions valued in $V(\Omega )$, so they are all internal functions$;$

\item the space of ultrafunctions $V_{\Lambda }(\Omega )$ is a vector space
of hyperfinite dimension;

\item if we identify a function $f$ with its natural extension $f^{\ast }$
then $V_{\Lambda }(\Omega )$ includes $V(\Omega ),$ hence we have that 
\begin{equation*}
V(\Omega )\subset V_{\Lambda }(\Omega )\subset V(\Omega ){^{\ast }.}
\end{equation*}
\end{enumerate}

\begin{remark}
\label{nina}Notice that the natural extension $f^{\ast }$ of a function $f$
is an ultrafunction if and only if $f\in V(\Omega ).$
\end{remark}

\begin{proof}The proof of this result is trivial\footnote{Any interested reader can find it in \cite{belu2013}.}. \end{proof}

Ultrafunctions can be used to give generalized solutions to some problems in
the calculus of variations (see e.g. \cite{milano}). Usually this kind of
problems have a "natural space" where to look for solutions: the appropriate
function space has to be a space in which the problem is well posed and
(relatively) easy to solve. For a very large class of problems the natural
space is a Sobolev space. However, many times even the best candidates to be
natural spaces are inadequate to study the problem, since there is no
solution in them. So the choice of the appropriate function space is part of
the problem itself; this choice is somewhat arbitrary and it might depend on
the final goals. In the framework of ultrafunctions this situation persists.
The general rule is: choose the "natural space" $V(\Omega )$ and look for a
generalized solution in $V_{\Lambda }(\Omega ).$ For many applications, an
hypothesis\footnote{%
E.g., in \cite{algebra} a (slightly modified) version of this hypothesis is
used to construct an embedding of the space of distributions in a particular
algebra of functions constructed by means of ultrafunctions.} that we need
to assume is that $D(\Omega )\subseteq V(\Omega )\subseteq L^{2}(\Omega ).$
In this case, since $V_{\Lambda }(\Omega )\subset \left[ L^{2}(\Omega)\right]
^{\ast },$ we can equip $V_{\Lambda }(\Omega )$ with the following scalar
product:%
\begin{equation}
\left( u,v\right) =\int^{\ast }u(x)v(x)\ dx,  \label{inner}
\end{equation}%
where $\int^{\ast }$ is the natural extension of the Lebesgue integral
considered as a functional%
\begin{equation*}
\int :L^{1}(\Omega )\rightarrow {\mathbb{R}}.
\end{equation*}%
The norm\footnote{%
Let us observe that both the scalar product and the norm take values in $%
\mathbb{R}^{\ast}$} of an ultrafunction will be given by 
\begin{equation*}
\left\Vert u\right\Vert =\left( \int^{\ast }|u(x)|^{2}\ dx\right) ^{\frac{1}{%
2}}.
\end{equation*}

Moreover, using the inner product (\ref{inner}), we can identify $%
L^{2}(\Omega )$ with a subset of $V^{\prime }(\Omega )$ and hence $\left[
L^{2}(\Omega )\right] ^{\ast }$ with a subset of $\left[ V^{\prime }(\Omega )%
\right] ^{\ast };$ in this case, $\forall f\in \left[ L^{2}(\Omega )\right]
^{\ast },\ $we let $\widetilde{f}$ be the unique ultrafunction such that, $%
\forall v\in V_{\Lambda }(\Omega ),$ 
\begin{equation*}
\int^{\ast} \widetilde{f}(x)v(x)\ dx=\int^{\ast} f(x)v(x)\ dx,
\end{equation*}%
namely we associate to every $f\in L^{2}(\Omega )^{\ast }$ the function $%
\widetilde{f}=P_{\Lambda }(f),$ where 
\begin{equation*}
P_{\Lambda }:\left[ L^{2}(\Omega )\right] ^{\ast }\rightarrow V_{\Lambda
}(\Omega )
\end{equation*}%
is the orthogonal projection.

\begin{remark}
There are a few different ways to prove the existence of an orthogonal
projection of $L^{2}(\Omega)^{\ast}$ on $V_{\Lambda}(\Omega)$. For example,
consider, for every $\lambda\in\mathfrak{L}$, the orthogonal projection $%
P_{\lambda}:L^{2}(\Omega)\rightarrow V_{\lambda}(\Omega)$. Let $%
F:=\lim_{\lambda \uparrow \Lambda } P_{\lambda}$. It is immediate to see
that $F:L^{2}(\Omega)^{\ast}\rightarrow V_{\Lambda}(\Omega)$ is an
orthogonal projection.
\end{remark}

Let us note that the key property to associate an ultrafunction to every
function in $\left[ L^{2}(\Omega )\right] ^{\ast }$ is that $\left[
L^{2}(\Omega )\right] ^{\ast }$ can be identified with a subset of $\left[
V^{\prime }(\Omega )\right] ^{\ast }.$ Therefore, using a similar idea, it
is also possible to extend a large class of operators:

\begin{definition}
Given an operator%
\begin{equation*}
\mathcal{A}:V(\Omega )\rightarrow V^{\prime }(\Omega ),
\end{equation*}%
we can extend it to an operator%
\begin{equation*}
\widetilde{\mathcal{A}}:V_{\Lambda }(\Omega )\rightarrow V_{\Lambda }(\Omega
)
\end{equation*}%
in the following way: given an ultrafunction $u,$ $\mathcal{A}_{\Lambda }(u)$
is the unique ultrafunction such that%
\begin{equation*}
\forall v\in V_{\Lambda }(\Omega ),\ \int^{\ast} \widetilde{\mathcal{A}}%
(u)vdx\ =\int^{\ast} \mathcal{A}^{\ast }(u)vdx;
\end{equation*}%
namely%
\begin{equation*}
\widetilde{\mathcal{A}}=P_{\Lambda }\circ \mathcal{A}^{\ast },
\end{equation*}%
where $P_{\Lambda }$ is the canonical projection.
\end{definition}

This association can be used, e.g., to define the derivative of an
ultrafunction, by setting 
\begin{equation*}
Du:=\widetilde{\partial }u=P_{\Lambda }(\partial ^{\ast }u)
\end{equation*}%
for every ultrafunction $u\in V_{\Lambda }(\Omega )\cap \mathcal{C}%
^{1}(\Omega )^{\ast }$.

\subsection{Applications to calculus of variations}

\label{APPLICAZIONE}

To give an example of application of ultrafunctions to calculus of
variations, we will show the ultrafunction interpretation of the Lavrentiev
phenomenon. Let us consider the following problem: minimize the functional

\begin{equation*}
J_{0}(u)=\int_{0}^{1}\left[ \left( \left\vert \nabla u\right\vert
^{2}-1\right) ^{2}+|u|^{2}\right] dx
\end{equation*}%
in the function space $\mathcal{C}_{0}^{1}(\Omega )=\mathcal{C}^{1}(\Omega
)\cap \mathcal{C}_{0}(\overline{\Omega }).$ We assume $\Omega $ to be
bounded to avoid problems of summability\footnote{%
This example has already been studied in greater detail in \cite{milano}.}.

It is not difficult to realize that any minimizing sequence $u_{n}$
converges uniformly to $0$ and that $J_{0}(u_{n})\rightarrow 0,$ but $%
J_{0}(0)>0$ for any $u\in \mathcal{C}_{0}^{1}(0,1)$. Hence there is no
minimizer in $\mathcal{C}^{1}_{0}(\Omega)$.

On the contrary, it is possible to show that this problem has a minimizer in
the space of ultrafunctions%
\begin{equation*}
V_{0}^{1}(\Omega )=\left[ \mathcal{C}^{1}(\Omega )\cap \mathcal{C}_{0}(%
\overline{\Omega })\right] _{\Lambda }.
\end{equation*}%
In $V_{0}^{1}(\Omega)$ our problem becomes%
\begin{equation}
\text{find }v\in V_{0}^{1}(\Omega )\text{ s.t. }\widetilde{J_{0}}(v)=%
\underset{u\in V_{0}^{1}(\Omega )}{\min }\widetilde{J_{0}}(u).  \tag{P}
\end{equation}%
To solve (P), let us prove the following "ultrafunction version" of an
existence result for minimizers of coercive continuous operators; the proof
is based on a variant of Faedo-Galerkin method.

\begin{theorem}
\label{B}Let $V(\Omega )\subseteq L^{2}(\Omega )$ be a vector space and let%
\begin{equation*}
J:V\left( \Omega \right) \rightarrow \mathbb{R}
\end{equation*}%
be an operator continuous and coercive on finite dimensional spaces. Then
the operator%
\begin{equation*}
\widetilde{J}:V_{\Lambda }\left( \Omega \right) \rightarrow \mathbb{R}^{\ast
}
\end{equation*}%
has a minimum point. If $J$ itself has a minimizer $u,$ then $u^{\ast }$ is
a minimizer of $\widetilde{J}.$
\end{theorem}

\begin{proof} Take $\lambda \in \mathfrak{L}$; since the operator 
\begin{equation*}
J|_{V_{\lambda }}:V_{\lambda }(\Omega )\longrightarrow \mathbb{R}
\end{equation*}

is continuous and coercive, it has a minimizer; namely%
\begin{equation*}
\exists u_{\lambda }\in V_{\lambda }\text{ }\forall v\in V_{\lambda }\text{ }%
J(u_{\lambda })\leq J(v).
\end{equation*}

We set%
\begin{equation*}
u_{\Lambda }=\lim_{\lambda \uparrow \Lambda }u_{\lambda }.
\end{equation*}

We show that $u_{\Lambda }$ is a minimizer of $\widetilde{J}.$ Let $v\in
V_{\Lambda }\left( \Omega \right) .$ Let us suppose that $v=\lim_{\lambda
\uparrow \Lambda }v_{\lambda };$ then by construction%
\begin{equation*}
\forall \lambda \in \mathfrak{L}\text{ }J(u_{\lambda })\leq J(v_{\lambda }),
\end{equation*}%
therefore%
\begin{equation*}
\widetilde{J}(u_{\Lambda })\leq \widetilde{J}(v).
\end{equation*}

If $J$ itself has a minimizer $\overline{u}$, then $u_{\lambda }$ is
eventually equal to $\overline{u}$ and hence $u_{\Lambda }=\overline{u}%
^{\ast }.$\end{proof} 

As a consequence, problem (P) has a solution, since the functional $J_{0}$
satisfies the hypothesis of Thm \ref{B}. So there exists an ultrafunction $%
u\in V_{0}^{1}(\Omega )$ that minimizes $\widetilde{J_{0}}.$ Moreover, it
can be represented as the $\Lambda $-limit of a function of minimizers of
the approximate problems on the spaces $\left[ \mathcal{C}^{1}(\Omega )\cap 
\mathcal{C}_{0}(\overline{\Omega })\right] _{\lambda }$. By using this
characterization, it is also possible to derive some qualitative properties
of $u$, e.g. it is not difficult to show that, $\forall x\in (0,1)^{\ast },$
the minimizer $u_{\Lambda }(x)\sim 0$ and that $\widetilde{J_{0}}(u_{\Lambda
})$ is a positive infinitesimal.

\end{document}